\documentclass[12pt]{amsart}
\usepackage{amsmath,amssymb,amsthm,mathtools}
\usepackage[dvipsnames]{xcolor}
\usepackage{enumitem}
\usepackage[hidelinks]{hyperref}
\hypersetup{
    pdftitle={A double-logarithmic upper bound on the chromatic number of the associahedron},
    pdfauthor={Sang-il Oum and David R. Wood}
}
\usepackage{tikz}
\usetikzlibrary{calc}
\tikzset{
  tnode/.style={circle, draw=black, fill=black, inner sep=0pt, minimum width=3pt},
  hl/.style={line width=4pt, opacity=0.35, line cap=round},
}
\newcommand{\tsubtree}[5]{\draw[hl,#4] ({#1-0.2},{#2-1.1}) -- (#1,#2);
  \draw (#1,#2) -- ({#1-0.55},{#2-1.1}) -- ({#1+0.55},{#2-1.1}) -- cycle;
  \node[font=\small] at ({#1+0.22},{#2-0.8}) {$#3$};
  \node[tnode,draw=#4,fill=#4,label={[text=#4]below:$#5$}] at ({#1-0.2},{#2-1.1}) {};
  \node[tnode] at (#1,#2) {};}
\usepackage[foot]{amsaddr}
\usepackage{lineno}

\usepackage[margin=30mm]{geometry}

\newtheorem{theorem}{Theorem}
\newtheorem{lemma}{Lemma}

\usepackage{zref-clever}
\zcsetup{cap=true}
\zcRefTypeSetup{claim}{Name-sg=Claim, name-sg=claim, Name-pl=Claims, name-pl=claims}

\renewcommand{\le}{\leqslant}
\renewcommand{\leq}{\leqslant}
\renewcommand{\ge}{\geqslant}
\renewcommand{\geq}{\geqslant}

\DeclareMathOperator{\parent}{parent}
\DeclareMathOperator{\ROOT}{root}
\DeclareMathOperator{\king}{king}
\newcommand{\HHH}{m}

\newcommand{\defn}[1]{\textcolor{Maroon}{\emph{#1}}}
\newcommand{\mathdefn}[1]{\textcolor{Maroon}{#1}}

\title[On the chromatic number of the associahedron]{A double-logarithmic upper bound on the~chromatic~number of the associahedron}
\author{Sang-il Oum$^\dagger$}
\address{$^\dagger$Discrete Mathematics Group, Institute for Basic Science (IBS), Daejeon,~South~Korea.}
\address{$^\dagger$Department of Mathematical Sciences, KAIST, Daejeon, South Korea.}
\email{sangil@ibs.re.kr}
\author{David R.~Wood$^\ddagger$}
\address{$^\ddagger$School of Mathematics, Monash University, Melbourne, Australia.}
\thanks{$^\dagger$Supported by the Institute for Basic Science (IBS-R029-C1).}
\thanks{$^\ddagger$Supported by the Australian Research Council and by NSERC.}
\email{david.wood@monash.edu}
\date{\today}

\begin{document}
\begin{abstract}
We show that the associahedron $\mathcal{A}_n$ has chromatic number $O(\log\log n)$, improving on the previously best known upper bound of $O(\log n)$.
\end{abstract}

\maketitle

\section{Introduction}

Associahedra are combinatorial objects that appear across mathematics, including 
in algebraic topology~\cite{Stasheff63a,Stasheff63b}, 
in order theory via the Tamari lattice~\cite{Tamari1951}, 
in enumerative combinatorics~\cite{StanleyCatalan}, 
in convex polytopes~\cite{Loday2004}, 
in discrete geometry~\cite{GKZ1994},  
in cluster algebras~\cite{FominZelevinsky2003}, and 
in monoidal categories~\cite{Kapranov1993}.

Fix an integer $n \ge 3$ and a convex polygon $P_n$ with $n$ labelled vertices $0,1,\dots,n-1$ in cyclic order. A \defn{diagonal} of $P_n$ is a segment joining a pair of non-consecutive vertices of $P_n$. A
\defn{triangulation} of $P_n$ is a maximal set of pairwise non-crossing diagonals. Every triangulation of $P_n$ has exactly $n-3$ diagonals, and divides $P_n$ into exactly $n-2$ triangles whose union is $P_n$. 
The \defn{associahedron} $\mathcal{A}_n$ (more precisely, its $1$-skeleton) is the graph whose vertices are the triangulations of $P_n$, where two triangulations~$T_1$ and~$T_2$ of~$P_n$ are adjacent in $\mathcal{A}_n$ if~$T_2$ can be obtained from~$T_1$ by deleting one diagonal of~$T_1$ and adding a different diagonal in~$T_2$ (so $T_1$ and $T_2$ have $n-4$ diagonals in common).

We study the chromatic number\footnote{A \defn{$k$-colouring} of a graph $G$ is a function that assigns one of $k$ colours to each vertex such that adjacent vertices are assigned distinct colours. The \defn{chromatic number $\chi(G)$} is the minimum integer $k$ such that $G$ is $k$-colourable.} of associahedra. This topic was initiated by Fabila-Monroy, Flores-Pe\~naloza, Huemer, Hurtado, Urrutia, and
Wood~\cite{FabilaEtAl2009}, who gave an explicit $\lceil n/2\rceil$-colouring of $\mathcal{A}_n$, and observed $\chi(\mathcal{A}_n) \in O(n/\log n)$ via a general theorem of Johansson~\cite{Johansson1994}
on triangle-free graphs. Addario-Berry, Reed,
Scott, and Wood~\cite{AddarioBerryEtAl} verified a conjecture of Fabila-Monroy et al.~\cite{FabilaEtAl2009}, stating that $\chi(\mathcal{A}_n) \in O(\log n)$. We improve this logarithmic bound to double-logarithmic.

\begin{theorem}
\label{thm:main}
$\chi(\mathcal{A}_n) \,\le\, 10\,\log_2\log_2 n$ for $n \ge 4$.
\end{theorem}

The proof of \zcref{thm:main} works with full binary trees $T$ rather than triangulations directly, using a
classical bijection recalled in \zcref{sec:bijection}. The heart of the
argument is a general-purpose scheme (\zcref{sec:forests}) for colouring trees under the so-called parent slide operation. 
This scheme is applied to the so-called tournament tree $\widehat{T}$ obtained by contracting a binary tree $T$ according to a ranking of its leaves (\zcref{sec:tournament}). Tree rotations in $T$ are encoded by parent slides in $\widehat{T}$ (together with a parity bit).

We now record other known graph-theoretic properties of~$\mathcal{A}_n$. By construction,  $|V(\mathcal{A}_n)|$ equals the Catalan number $C_{n-2} = \frac{1}{n-1}\binom{2n-4}{n-2}$~\cite{StanleyCatalan}, and  $\mathcal{A}_n$ is $(n-3)$-regular. Moreover, it follows from Balinski's theorem~\cite{Balinski1961} that $\mathcal{A}_n$ has connectivity equal to $n-3$, as proved directly by Hurtado
and Noy~\cite{HurtadoNoy1999}.
The diameter of $\mathcal{A}_n$ was studied by Sleator, Tarjan, and Thurston~\cite{SleatorTarjanThurston1988}, who
related it to volumes of ideal polyhedra in hyperbolic $3$-space and showed it equals
$2n-10$ for all sufficiently large~$n$.
Pournin~\cite{Pournin2014} showed the diameter equals $2n-10$ for every $n>12$. 
Lucas~\cite{Lucas1987} showed that $\mathcal{A}_n$ is Hamiltonian for~$n\geq 5$; Hurtado and Noy~\cite{HurtadoNoy1999} later gave a simple proof. Lee~\cite{Lee1989} determined the automorphism group of $\mathcal{A}_n$.  Parlier and Petri~\cite{ParlierPetri2018} proved bounds on the genus of $\mathcal{A}_n$.

\subsection{Rotations in Binary Trees}
\label{sec:bijection}

We now recall the connection between triangulations of~$P_n$ and rooted binary trees. Since we will consider graphs whose vertices are themselves trees, we call the vertices of a tree \defn{nodes}, reserving the word `vertex' for the graphs $\mathcal{A}_n$,~$\mathcal{R}_L$, and~$\mathcal{G}_{V,m}$ defined below. A \defn{rooted tree} $T$ is a tree with a distinguished \defn{root} node. 
Let $\parent_T(u)$ denote the \defn{parent} of each non-root node $u$ in $T$. 
We call~$u$ a \defn{child} of $v$ if $\parent_T(u) = v$. A \defn{leaf} is a non-root node with no children. A \defn{full binary tree} is a rooted tree  in which every non-leaf node $v$ has exactly
two children, distinguished as the \defn{left child} and \defn{right child} of $v$; the subtree rooted at the left/right child of $v$ is called
the \defn{left}/\defn{right subtree} of $v$. 

A full binary tree $T$ with $L\geq 2$ leaves  has exactly $L-1$ non-leaf nodes. Label the leaves $1,\dots,L$ such that at each non-leaf node $v$ of $T$:
\begin{itemize}
\item the leaves in the subtree of $T$ rooted at $v$ form a set of consecutive integers, and 
\item the leaves in the left subtree of $v$ receive smaller labels than the leaves in the right subtree of $v$.
\end{itemize}

We now describe the well-known bijection between triangulations and full binary trees. 
Fix the boundary edge $e=\{n-1,0\}$ of $P_n$ as a distinguished `root edge.' Every
triangulation $T$ of~$P_n$ contains a unique triangle using the edge $e$; let its third
vertex be $k \in \{1,\dots,n-2\}$. This triangle splits the rest of the polygon into two
parts, the sub-polygon on vertices $0,1,\dots,k$ and the sub-polygon on vertices
$k,k+1,\dots,n-1$, each triangulated by the diagonals of~$T$ restricted to it (together with
the missing side $\{0,k\}$ and $\{k,n-1\}$ respectively, if that sub-polygon has more than two 
vertices). We turn this into a full binary tree with $n-1$ leaves recursively, where a
sub-polygon with two vertices (a single edge) becomes a single leaf, and a sub-polygon with more
vertices becomes a non-leaf node whose left and right children are the (recursively built) trees for its two parts. This is a bijection between triangulations of~$P_n$ and full binary trees with $n-1$ leaves; see \zcref{fig:bijection} for an example. Throughout the rest of this paper we work with full binary trees, and set $L = n-1$ for the number of leaves.

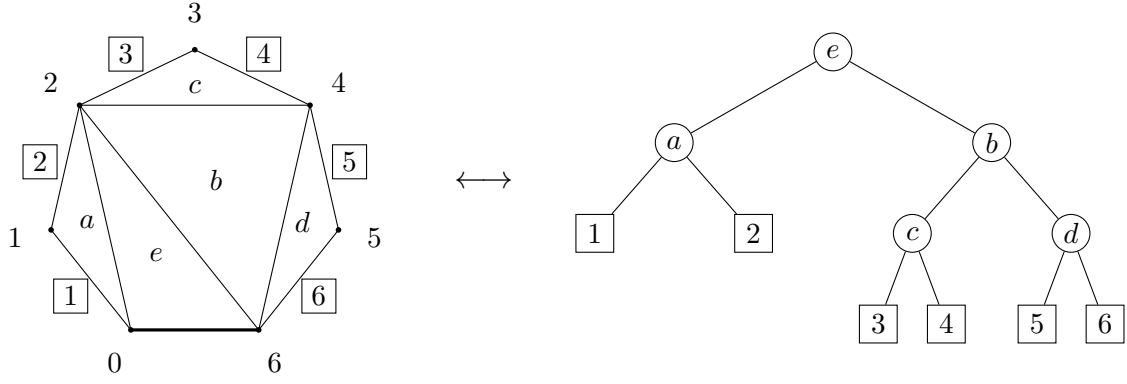
\begin{figure}[t]
\[
\begin{array}{c@{\qquad}c@{\qquad}c}
\begin{tikzpicture}[
  scale=0.75,
  baseline=(current bounding box.center),
  every node/.style={font=\small},
  lf/.style={draw,rectangle,fill=white,inner sep=0pt,minimum size=4.5mm},
  tr/.style={fill=white,inner sep=1pt} 
]
  \foreach \i in {0,...,6} {\coordinate (P\i) at ({244.2857-51.42857*\i}:2.6);}
  \draw (P0)--(P1)--(P2)--(P3)--(P4)--(P5)--(P6);
  \draw[very thick] (P6)--(P0);
  \draw (P0)--(P2);  \draw (P2)--(P6);  \draw (P2)--(P4);  \draw (P4)--(P6);
  \foreach \i in {0,...,6} {\fill (P\i) circle (1.4pt);}
  \foreach \i in {0,...,6} {\node at ({244.2857-51.42857*\i}:3.25) {$\i$};}
  \node[lf] at (218.571:2.80) {$1$};
  \node[lf] at (167.143:2.80) {$2$};
  \node[lf] at (115.714:2.80) {$3$};
  \node[lf] at ( 64.286:2.80) {$4$};
  \node[lf] at ( 12.857:2.80) {$5$};
  \node[lf] at (-38.571:2.80) {$6$};
  \node[tr] at (-0.678,-1.021) {$e$};
  \node[tr] at (-1.899,-0.433) {$a$};
  \node[tr] at ( 0.376, 0.300) {$b$};
  \node[tr] at ( 0.000, 1.947) {$c$};
  \node[tr] at ( 1.899,-0.433) {$d$};
\end{tikzpicture}
&
\longleftrightarrow
&
\begin{tikzpicture}[
  baseline=(current bounding box.center),
  level distance=12mm,
  level 1/.style={sibling distance=42mm},
  level 2/.style={sibling distance=21mm},
  level 3/.style={sibling distance=9mm},
  every node/.style={font=\small},
  inode/.style={draw,circle,inner sep=0pt,minimum size=5mm},
  lnode/.style={draw,rectangle,inner sep=0pt,minimum size=5mm}
]
\node[inode] {$e$}
  child { node[inode] {$a$}
    child { node[lnode] {$1$} }
    child { node[lnode] {$2$} }
  }
  child { node[inode] {$b$}
    child { node[inode] {$c$}
      child { node[lnode] {$3$} }
      child { node[lnode] {$4$} }
    }
    child { node[inode] {$d$}
      child { node[lnode] {$5$} }
      child { node[lnode] {$6$} }
    }
  };
\end{tikzpicture}
\end{array}
\]
\caption{The bijection for $n=7$, applied to the triangulation with diagonals $\{0,2\}$,
$\{2,4\}$, $\{4,6\}$, and $\{2,6\}$. Each sub-polygon arising in the recursion is recorded
here by the triangle sitting on its base edge. The root $e$ is the triangle on the root
edge $\{6,0\}$ (drawn thick), and the non-leaf nodes $a$, $b$, $c$, and $d$ (circles in the
tree) are the triangles on the diagonals $\{0,2\}$, $\{2,6\}$, $\{2,4\}$, and $\{4,6\}$,
respectively. The two-vertex sub-polygons are the remaining boundary edges, giving the
leaves $1,\dots,6$ (squares) in order.}
\label{fig:bijection}
\end{figure}

We now interpret the edges of $\mathcal{A}_n$ in terms of rotations. 
Let $T$ be a full binary tree containing a non-leaf node $v$ whose right child is itself a non-leaf. Let $A$ be the left subtree of $v$, and let $B$ and $C$ be the left and right subtrees of the right child of $v$. We say the \defn{shape} at $v$ is 
$\big(A,(B,C)\big)$.  As illustrated in \zcref{fig:Rotations}, the \defn{left rotation} at $v$ replaces this shape by
$\big((A,B),\,C\big)$; 
that is, $C$ becomes the right subtree of $v$, and $A$ and $B$ become the left and right subtrees of the left child of $v$. Everything outside this local picture is left unchanged.
A \defn{right rotation} is the reverse move (applied at a node whose left child is a non-leaf).

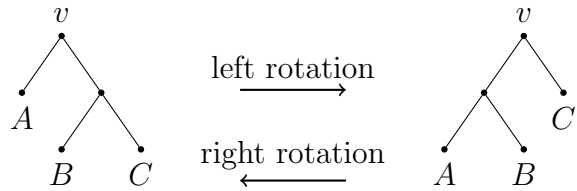
\begin{figure}[!h]
\[
\begin{array}{ccc}
\begin{tikzpicture}[level distance=10mm, sibling distance=14mm, scale=0.75]
  \tikzstyle{v}=[circle, draw=black, solid, fill=black, inner sep=0pt, minimum width=2pt]
\node [v,label=$v$] {}
 child { node [label=below:$A$,v]{} }
 child { node [v] {}
   child { node [label=below:$B$,v] {} }
   child { node[label=below:$C$,v] {} }
 };
\end{tikzpicture}
&
\begin{tikzpicture}[baseline=0mm]
\draw[->,thick] (0,1.4) -- (1.4,1.4) node[midway,above]{left rotation};
\draw[<-,thick] (0,0.2) -- (1.4,0.2) node[midway,above]{right rotation };
\end{tikzpicture}
&
\begin{tikzpicture}[level distance=10mm, sibling distance=14mm, scale=0.75]
    \tikzstyle{v}=[circle, draw=black, solid, fill=black, inner sep=0pt, minimum width=2pt]

\node [label=$v$,v]{}
 child { node [v]{}
   child { node[label=below:$A$,v] {} }
   child { node[label=below:$B$,v] {} }
 }
 child { node [label=below:$C$,v] {} };
\end{tikzpicture}
\end{array}
\]
\caption{Left and right rotations in full binary trees.}
\label{fig:Rotations}
\end{figure}

A diagonal flip in a triangulation of $P_n$ replaces one diagonal of a quadrilateral by the other, and under the recursive correspondence above, a diagonal of the triangulation
corresponds to an edge of the binary tree joining a non-leaf node to one of its non-leaf children,
with the quadrilateral on either side of that diagonal corresponding exactly to the shape $\big(A,(B,C)\big)$ or $\big((A,B),C\big)$ used in the definition of rotation; flipping the diagonal is exactly performing the corresponding rotation, and every
other diagonal (hence every other part of the tree) is unaffected. Thus, two triangulations of $P_n$ are adjacent in $\mathcal{A}_n$ if and only if one of the corresponding full binary trees is obtained from the other by a rotation.

\section{A General Framework}\label{sec:forests}

We now set up a general framework that we will later specialise to associahedra. 
Let~$F$ be a rooted tree on a node set $V$, and let $v,b,c \in V$ satisfy $\parent_F(v) = b$ and $\parent_F(b) = c$ (so $c$ is the parent of the parent of $v$). As illustrated in \zcref{fig:Slide}, the \defn{parent slide at $v$ from $b$ to $c$} is the rooted tree $F'$ on $V$, with the same root as $F$, obtained by setting $\parent_{F'}(v) = c$ and leaving every other parent unchanged.

\begin{figure}[!h]
\[
\begin{array}{ccc}
\begin{tikzpicture}[level distance=10mm, sibling distance=10mm, scale=0.85]
    \tikzstyle{v}=[circle, draw=black, solid, fill=black, inner sep=0pt, minimum width=2pt]
\node [v,label=left:$c$]{}
  child { node [v,label=left:$b$] {}
    child { node [v,label=left:$v$]{} }
  };
\end{tikzpicture}
&
\begin{tikzpicture}[baseline=-5mm]
\draw[->,thick] (0,0) -- (1.4,0) node[midway,above]{slide at $v$};
\end{tikzpicture}
&
\begin{tikzpicture}[scale=0.85]
      \tikzstyle{v}=[circle, draw=black, solid, fill=black, inner sep=0pt, minimum width=2pt]
\node at (0,0)[label=left:$c$,v](c) {};
\node at (0,-1)[label=left:$b$,v](b){};
\node at (0,-2) [label=left:$v$,v](v){};
\draw (b)--(c);
\draw (v) [bend right=40] to (c);
\end{tikzpicture}
\\
F & & F'
\end{array}
\]
\caption{The parent slide at $v$ from $b$ to $c$.}
\label{fig:Slide}
\end{figure}
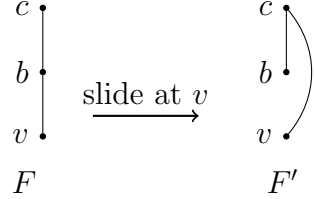

A parent slide is analogous to the `pointer jumping' technique from parallel algorithms, used for instance in parallel list ranking by Wyllie~\cite{Wyllie1979}. 

For $m\geq 1$, an \defn{$m$-labelling} of a tree $F$ is a colouring $\ell\colon V(F) \to \{0,1,\dots,m-1\}$. (We use `labelling' to distinguish from colourings of other graphs.)\ 
A \defn{rooted $m$-labelled tree} is a pair $(F,\ell)$ of 
a rooted tree $F$ 
and an $m$-labelling of~$F$.
For a finite set $V$ and $m \ge 1$, let $\mathcal{G}_{V,m}$ be the graph whose vertices are all rooted $m$-labelled trees on $V$, where two rooted $m$-labelled trees $(F,\ell)$ and $(F',\ell')$ are adjacent in $\mathcal{G}_{V,m}$ whenever $\ell = \ell'$ and $F'$ is a parent slide of $F$ or vice versa.

Consider an edge of $\mathcal{G}_{V,m}$ joining $(F,\ell)$ and $(F',\ell)$, given by a parent slide
at $v$ from~$b$ to~$c$. Since $\parent_F(v)=b$ and $\parent_F(b)=c$ we have $\ell(v)\neq\ell(b)$ and $\ell(b)\neq \ell(c)$. Since $\parent_{F'}(v)=c$ we have $\ell(v)\neq\ell(c)$. Thus 
$\ell(v),\ell(b),\ell(c)$ are pairwise distinct; we refer to this as the \defn{three-label property}. 

As a warm-up, we have the following simple bound on the chromatic number of $\mathcal{G}_{V,m}$.

\begin{lemma}\label{lem:basecolor}
For any finite set $V$ and integer $m\ge1$, 
$\mathcal{G}_{V,m}$ is $m$-colourable.
\end{lemma}
\begin{proof}
If $m=1$ then the claim is trivial since 
$\mathcal{G}_{V,m}$ has no edges. Now assume that $m\geq 2$. 
For each vertex $(F,\ell)$ of $\mathcal{G}_{V,m}$, let 
\[
\phi(F,\ell) \;:=\; \Big( \sum_{u: \text{non-root node of }F} \ell(\parent_F(u)) \Big) \bmod m.
\]
We claim that $\phi$ is a colouring of $\mathcal{G}_{V,m}$. 
Consider an edge of $\mathcal{G}_{V,m}$ joining $(F,\ell)$ and $(F',\ell)$, given by a parent slide at
$v$ from $b$ to $c$. The sum defining $\phi$ changes only in the term indexed by $v$ (since every other non-root node keeps the same parent), and that term changes from $\ell(b)$ to $\ell(c)$. Since $\ell(b)\not\equiv \ell(c)\pmod m$ (by the three-label property), 
the change  is non-zero modulo $m$. Hence $\phi(F,\ell) \ne \phi(F',\ell)$.
\end{proof}

In fact, we can save one colour by more careful analysis as follows.
\begin{lemma}\label{lem:basepalette}
For any finite set $V$ and integer $m\ge2$, 
$\mathcal{G}_{V,m}$ is $(m-1)$-colourable.
\end{lemma}
\begin{proof}
If $m=2$, then no parent slide can occur (by the three-label property), implying 
$\mathcal{G}_{V,m}$ has no edges and is thus 1-colourable. Now assume that $m\geq 3$. 
For each fixed value $x \in \{0,\dots,m-1\}$, define a bijection $f_x$ from
$\{0,\dots,m-1\}\setminus\{x\}$ to $\{0,\dots,m-2\}$ by
\[
f_x(y) := 
\begin{cases} 
y & \text{if }y < x\\ 
y-1 & \text{if }y > x.
\end{cases}
\]
($f_x$ shifts every label above $x$ down by $1$.) For each vertex $(F,\ell)$ of $\mathcal{G}_{V,m}$, define
\[
\phi((F,\ell)) := \Big( \sum_{u:\text{ non-root node of }F} f_{\ell(u)}\big(\ell(\parent_F(u))\big) \Big)\bmod{(m-1)}.
\]
Consider an edge of $\mathcal{G}_{V,m}$ joining two rooted $m$-labelled trees $(F,\ell)$ and $(F',\ell)$, given by a parent slide at $v$ from $b$ to $c$. 
Then $\ell(v)$ does not change (only
the parent of $v$ changes), so the same bijection $f_{\ell(v)}$ is used before and
after; since $\ell(b) \ne \ell(c)$ and both differ from~$\ell(v)$ (by the three-label property), and
$f_{\ell(v)}$ is injective, $f_{\ell(v)}(\ell(b)) \ne f_{\ell(v)}(\ell(c))$. Thus the only
changed term in the sum changes to a different value in $\{0,\dots,m-2\}$. This term changes by a non-zero integer of absolute value at most $m-2$, so the sum changes by a non-zero amount modulo $m-1$. Thus $\phi((F,\ell)) \neq \phi((F',\ell))$, and $\phi$ is a $(m-1)$-colouring of $\mathcal{G}_{V,m}$.
\end{proof}

\zcref{lem:basepalette} will provide the base case in our recursive colouring of $\mathcal{G}_{V,m}$, which trades one use of a large alphabet for two uses of a much smaller alphabet, at the cost of a single extra bit.

\label{def:rofm}
For $m \ge 1$, let
\[
r_m := \min\big\{ r\in\mathbb Z : r\ge 2\text{ and }\tbinom{r}{2} \ge m \big\} 
= \left\lceil \tfrac12 (1+\sqrt{1+8m}) \right\rceil.
\]
Assign to each $x\in\{ 0,\dots,m-1\}$ a distinct $2$-element subset 
$S_x = \{s^0_x, s^1_x\} \subseteq  \{0,\dots,r_m-1\}$ such that $s^0_x < s^1_x$.

Fix a vertex $(F,\ell)$ of $\mathcal{G}_{V,m}$. The following notation suppresses~$\ell$, although it depends on~$\ell$ as well as on~$F$. Define a function
$\ell'_F\colon V \to \{0,\dots,r_m-1\}$ by
\[
\ell'_F(u) := 
\begin{cases}
s^0_{\ell(u)} & \text{if $u$ is the root of $F$},\\
\min\big(S_{\ell(u)} \setminus S_{\ell(\parent_F(u))}\big) & 
\text{if $u$ is a non-root node of $F$.}
\end{cases}
\]
Since the assignment $x\mapsto S_x$ is injective and $\ell(u) \neq \ell( \parent_F(u) )$, in this definition, $S_{\ell(u)} \ne S_{\ell(\parent_F(u))}$, so $S_{\ell(u)}\setminus S_{\ell(\parent_F(u))} \neq\emptyset$. Also define $\mathdefn{\varepsilon_F(u)} \in \{0,1\}$ by
$\ell'_F(u) = s^{\varepsilon_F(u)}_{\ell(u)}$; that is,  $\varepsilon_F(u)$ records which of the two elements of $S_{\ell(u)}$ was chosen as the new label of  $u$.

\begin{lemma}\label{lem:detector}
$\ell'_F$ is an $r_m$-labelling of the tree $F$, implying   $(F,\ell'_F)$ is a vertex of
$\mathcal{G}_{V,r_m}$. Moreover, if $(F,\ell)$ and $(F',\ell)$ are adjacent in $\mathcal{G}_{V,m}$ given
by a parent slide at $v$ in $F$, then 
$\ell'_F(u)=\ell'_{F'}(u)$ for all $u \neq v$, and if $\ell'_F(v)\neq \ell'_{F'}(v)$ then $\varepsilon_F(v) \ne \varepsilon_{F'}(v)$.
\end{lemma}
\begin{proof}
Let $u$ be a non-root node of~$F$ with parent $w=\parent_F(u)$. By definition, $\ell'_F(u) \in
S_{\ell(u)}\setminus S_{\ell(w)}$, so in particular $\ell'_F(u) \notin S_{\ell(w)}$. On the
other hand $\ell'_F(w) \in S_{\ell(w)}$ always, whether $w$ is a root (where
$\ell'_F(w)=s^0_{\ell(w)} \in S_{\ell(w)}$ by definition) or not (where $\ell'_F(w) \in
S_{\ell(w)}\setminus(\cdots) \subseteq S_{\ell(w)}$). Thus $\ell'_F(u) \ne \ell'_F(w)$. 
Hence $\ell'_F$ is an $r_m$-labelling of $F$.

The formula for $\ell'_F(u)$ depends only on $\ell(u)$ (which never changes, since the $m$-labelling~$\ell$ is the same at both ends of an edge of $\mathcal{G}_{V,m}$) and on $\ell(\parent_F(u))$. A
parent slide changes $\parent_F(u)$ only for $u=v$; for every other node $u$,
$\ell(\parent_F(u))=\ell(\parent_{F'}(u))$ trivially since $\parent_F(u)=\parent_{F'}(u)$, so
$\ell'_F(u)=\ell'_{F'}(u)$.

Both possible values of $\ell'_F(v)$ (before and after the slide) lie in the fixed two-element set~$S_{\ell(v)}$ (since the slide does not change $\ell(v)$). If $\ell'_F(v) \ne \ell'_{F'}(v)$, then they are the two different elements of the same $2$-element set, so by definition of $\varepsilon$, exactly one of $\varepsilon_F(v),\varepsilon_{F'}(v)$ is $0$ and the other is $1$.
\end{proof}

\begin{lemma}
\label{lem:NewLemma}
For any finite set $V$ and integer $m\ge1$, 
$\chi(\mathcal{G}_{V,m}) \leq 2 \chi( \mathcal{G}_{V,r_m})$.
\end{lemma}
\begin{proof}
Let $\phi$ be a $k$-colouring of $\mathcal{G}_{V,r_m}$. 
For each vertex $(F,\ell)$ of $\mathcal{G}_{V,m}$, let
\[
\phi'((F,\ell)) \;:=\; \Big(\; \phi((F, \ell'_F))\;,\;\; \big( \sum_{u \in V} \varepsilon_F(u) \big) \bmod 2 \;\Big).
\]
This is a legitimate use of $\phi$, since $(F,\ell'_F)$ is a vertex of
$\mathcal{G}_{V,r_m}$ by \zcref{lem:detector}. We now show that $\phi'$ is a colouring. 
Let $(F,\ell)$ and $(F',\ell)$ be adjacent in $\mathcal{G}_{V,m}$, given by a slide in $F$ at $v$ from~$b$ to~$c$. 
If $\sum_u \varepsilon_F(u) \not\equiv \sum_u \varepsilon_{F'}(u) \pmod 2$, then the two colours differ in their second coordinate. Otherwise, the two sums agree. For every $u \neq v$, \zcref{lem:detector} gives
$\ell'_F(u) = \ell'_{F'}(u)$, and hence $\varepsilon_F(u) = \varepsilon_{F'}(u)$,
since $\varepsilon_F(u)$ is determined by $\ell(u)$ and $\ell'_F(u)$.
Since the two sums agree, it follows that $\varepsilon_F(v) = \varepsilon_{F'}(v)$.
By the last part of \zcref{lem:detector}, this implies $\ell'_F(v) = \ell'_{F'}(v)$.
Thus $\ell'_F = \ell'_{F'}$, so $(F,\ell'_F)$ and $(F',\ell'_{F'})$ are two
vertices of~$\mathcal{G}_{V,r_m}$ with the same labelling that are related by a parent
slide at $v$. Hence they are adjacent in $\mathcal{G}_{V,r_m}$, and
$\phi((F,\ell'_F)) \neq \phi((F',\ell'_{F'}))$. In both cases, 
$(F,\ell)$ and $(F',\ell)$ are assigned distinct colours, 
so $\phi'$ is a colouring of $\mathcal{G}_{V,m}$. Thus $\chi(\mathcal{G}_{V,m}) \leq 2k$.
\end{proof}
\section{The Tournament Tree}\label{sec:tournament}

We now specialise the machinery in \zcref{sec:forests} to associahedra, by building, from any full binary tree $T$, a specific rooted tree on the leaf set of~$T$, together with one extra bit of information, such that a rotation of $T$ changes this data in a way the machinery of  \zcref{sec:forests} can detect. Throughout this section, we fix a positive integer $L$. An \defn{interval} is a nonempty set of consecutive integers in $\{1,\dots,L\}$. Two disjoint intervals $I$ and $J$ are \defn{adjacent} if $I\cup J$ is an interval. For a positive integer $m$, an \defn{$m$-ranking} of $\{1,2,\ldots,L\}$ is a function $\rho\colon\{1,\dots,L\}\to\{0,\dots,m-1\}$ such that every interval $I$ has a unique element of maximum value under~$\rho$; we call this element the \defn{king} of $I$ and denote it by \mathdefn{$\king(I)$}. 

\begin{lemma}\label{lem:nu-king}
  There exists a $\lceil\log_2(L+1)\rceil$-ranking of $\{1,2,\ldots,L\}$.
\end{lemma}
\begin{proof}
  For a positive integer $i$, 
  let $\rho(i)$ be the maximum integer such that $2^{\rho(i)}$ divides $i$ (so $i$ equals $2^{\rho(i)}$ times an odd number).
  For each $i\in\{1,\dots,L\}$, we have 
$2^{\rho(i)}\le i\le L$, and therefore $\rho(i)<\log_2(L+1)$.

We claim that every interval $I$ has a unique element of maximum $\rho$-value. 
Suppose that distinct integers in $I$ both attain the maximum $\rho$-value $e$ of the numbers in $I$. Choose such integers $i,j$ such that $i<j$ and $j-i$ is minimum. 
So $i=2^ea$ and $j=2^eb$ for some odd integers $a<b$. So $a+b$ is even, and $\frac{i+j}{2}=2^{e}(\frac{a+b}{2})$ is an integer in $I$ whose $\rho$-value is at least~$e$.
If it exceeds $e$, that contradicts maximality and if it equals $e$, it contradicts the minimum choice of $j-i$, which contradicts the choice of  $i$ and $j$. 
\end{proof}
\begin{lemma}\label{lem:ranking}
Let $\rho$ be an $m$-ranking. For all adjacent intervals $I$ and $J$, 
\[
\king(I\cup J)\in\{\king(I),\king(J)\}\quad\text{and}\quad\rho(\king(I))\neq\rho(\king(J)).
\]
\end{lemma}
\begin{proof}
The maximum value of $\rho$ on $I\cup J$ is attained at $\king(I)$ or at $\king(J)$, so $\king(I\cup J)\in\{\king(I),\king(J)\}$. If $\rho(\king(I))=\rho(\king(J))$, then the two distinct elements $\king(I)$ and $\king(J)$ both attain the maximum value of $\rho$ on $I\cup J$, contradicting the uniqueness of $\king(I\cup J)$.
\end{proof}

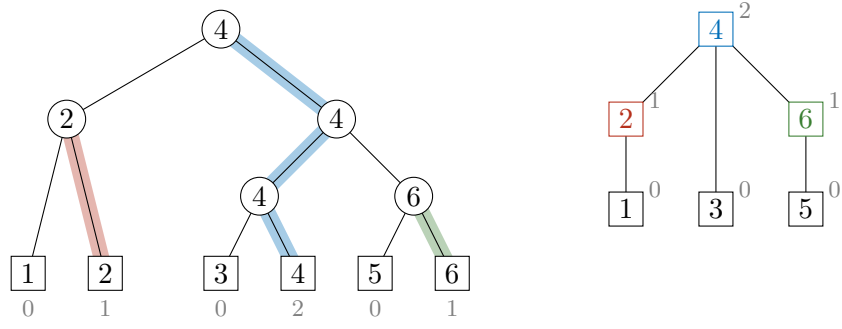
\begin{figure}[t]
\centering
\begin{tikzpicture}[baseline=0pt,scale=0.85,
  inode/.style={draw,circle,inner sep=0pt,minimum size=5mm,font=\small},
  lnode/.style={draw,rectangle,inner sep=0pt,minimum size=4.5mm,font=\small}]
  \coordinate (e) at (0,0);      \coordinate (a) at (-2.4,-1.4);
  \coordinate (b) at (1.8,-1.4); \coordinate (c) at (0.6,-2.6);
  \coordinate (d) at (3.0,-2.6);
  \coordinate (l1) at (-3.0,-3.8); \coordinate (l2) at (-1.8,-3.8);
  \coordinate (l3) at (0.0,-3.8);  \coordinate (l4) at (1.2,-3.8);
  \coordinate (l5) at (2.4,-3.8);  \coordinate (l6) at (3.6,-3.8);
  \draw[hl,line width=6pt,NavyBlue] (l4) -- (c) -- (b) -- (e);
  \draw[hl,line width=6pt,BrickRed] (l2) -- (a);
  \draw[hl,line width=6pt,OliveGreen] (l6) -- (d);
  \draw (e) -- (a) -- (l1) (a) -- (l2) (e) -- (b) -- (c) -- (l3) (c) -- (l4) (b) -- (d) -- (l5) (d) -- (l6);
  \node[inode,fill=white] at (e) {$4$};
  \node[inode,fill=white] at (a) {$2$};
  \node[inode,fill=white] at (b) {$4$};
  \node[inode,fill=white] at (c) {$4$};
  \node[inode,fill=white] at (d) {$6$};
  \foreach \i/\r in {1/0,2/1,3/0,4/2,5/0,6/1} {
    \node[lnode,fill=white] at (l\i) {$\i$};
    \node[font=\scriptsize,text=gray] at ($(l\i)+(0,-0.55)$) {$\r$};
  }
\end{tikzpicture}
\hspace{0.1\textwidth}
\begin{tikzpicture}[baseline=0pt,scale=0.85,
  lnode/.style={draw,rectangle,inner sep=0pt,minimum size=4.5mm,font=\small}]
  \coordinate (h4) at (0,0);
  \coordinate (h2) at (-1.4,-1.4); \coordinate (h3) at (0,-2.8); \coordinate (h6) at (1.4,-1.4);
  \coordinate (h1) at (-1.4,-2.8); \coordinate (h5) at (1.4,-2.8);
  \draw (h4) -- (h2) -- (h1) (h4) -- (h3) (h4) -- (h6) -- (h5);
  \foreach \i/\r/\col in {4/2/NavyBlue,2/1/BrickRed,6/1/OliveGreen,3/0/black,1/0/black,5/0/black} {
    \node[lnode,fill=white,draw=\col,text=\col] at (h\i) {$\i$};
    \node[font=\scriptsize,text=gray] at ($(h\i)+(0.45,0.3)$) {$\r$};
  }
\end{tikzpicture}
\caption{
Left: A full binary tree $T$. Each non-leaf node is labelled by its king, and the paths $X_2$ (red), $X_4$ (blue), and $X_6$ (green) are highlighted; every other $X_v$ is the single leaf~$v$. The ranking values $(\rho(1),\dots,\rho(6))=(0,1,0,2,0,1)$ are shown in grey. 
Right: The tournament tree $\widehat{T}$ obtained from~$T$ by contracting each $X_v$ to the node $v$.}
\label{fig:tournament}
\end{figure}

In the remainder of this section, we fix a positive integer $m$ and an $m$-ranking $\rho$ of $\{1,\dots,L\}$. Let $T$ be a full binary tree with $L$ leaves. 
Recall that the leaves of $T$ are numbered $1,\dots,L$. For each node $w$ of $T$, the leaves in the subtree of $T$ rooted at $w$ form an interval, and we let \mathdefn{$\king_T(w)$} be the king of this interval. So $\king(w)$ is a leaf in the subtree rooted at $w$. If $w$ is not a leaf, then the sets of leaves in the subtrees rooted at the two children of~$w$ are adjacent intervals, and therefore $\king(w)$ is the king of one of the two children of~$w$ by \zcref{lem:ranking}. For a leaf $v$ of $T$, let 
\[\mathdefn{X^T_v} := \{\, w \in V(T) : \king_T(w) = v \,\}.\] 
When $T$ is clear from the context, we write $\king(w)$ and $X_v$ for $\king_T(w)$ and $X^T_v$, and similarly for $a_v$ below. 
By definition, $v\in X_v$ and every other node in $X_v$ is an ancestor of~$v$. Moreover, if $w\in X_v\setminus\{v\}$, then the child of $w$ whose subtree contains~$v$ is also in $X_v$, since $\king(w)=v$ is the king of a child of $w$ whose subtree contains $v$. So
$X_v$ induces a path from~$v$ to some ancestor~$a_v=a^T_v$ of $v$ in~$T$. 
By construction, 
$\{X_v: v\in\{1,2,\ldots,L\}\}$ is a partition of $V(T)$. 
Define the \defn{tournament tree}~$\widehat{T}$ to be the rooted tree on the node set $\{1,\dots,L\}$ obtained from~$T$ by contracting each $X_v$ to a single node labelled $v$; see \zcref{fig:tournament}. 
We call $a_v$ the \defn{top} of $X_v$. Equivalently, the root of $\widehat{T}$ is $\king(\ROOT(T))$, and for each leaf $u$ with $a_u\neq\ROOT(T)$, 
\[\parent_{\widehat{T}}(u)=\king(\parent_T(a_u)).\]

\begin{lemma}\label{lem:PT-proper}
The ranking $\rho$ is an $m$-labelling of the rooted tree $\widehat{T}$.
\end{lemma}

\begin{proof}
Suppose $\parent_{\widehat{T}}(u)=w$. Let $q$ be the parent of $a_u$ in $T$, so $\king(q)=w$, and let $s$ be the other child of $q$. Since $a_u$ is the top of $X_u$, we have $\king(q)\neq u=\king(a_u)$, and therefore $\king(s)=\king(q)=w$. The sets of leaves in the subtrees rooted at $a_u$ and $s$ are adjacent intervals, whose kings are $u$ and $w$. By \zcref{lem:ranking}, $\rho(u)\neq\rho(w)$.
\end{proof}

We now analyse how a rotation in $T$ changes $\widehat{T}$. For a leaf $u$ of $T$, let $(u=u_0,u_1,\dots,u_k=a_u)$ be the path $X_u$. For $1 \le i \le k$, $u_i$ has two children, one of which is $u_{i-1}$; let
$\omega_u(T) \in \{L,R\}^k$ be the string whose $i$-th letter is $R$ if $u_i$'s other child (the one not equal to $u_{i-1}$) is $u_i$'s right child, and $L$ if it is $u_i$'s left
child. Call this the \defn{contraction word} of $u$.

For a string $w \in \{L,R\}^*$, let \defn{$\mathrm{inv}(w)$} be the number of pairs of positions $i<j$ with $w_i=R$ and $w_j = L$ (an $R$ before an $L$ is an
``inversion''). Define 
\[
i(T) = \sum_{u=1}^{L} \mathrm{inv}(\omega_u(T)) \pmod 2.
\]

The following lemma is trivial.
\begin{lemma}\label{lem:adjtrans}
For any strings $X,Y \in \{L,R\}^*$, if  $w = X\,RL\,Y$ and $w' = X\,LR\,Y$, then $\mathrm{inv}(w) =
\mathrm{inv}(w')+1$.
\end{lemma}

\begin{lemma}
\label{TT}
Let $T_1$, $T_2$ be full binary trees on $L$ leaves indexed by $1,2,\ldots,L$.
If $T_2$ is obtained from $T_1$ by a rotation, then 
either $(\widehat{T}_1,\rho)$ and $(\widehat{T}_2,\rho)$ are adjacent in $\mathcal{G}_{\{1,\dots,L\},m}$, or $\widehat{T}_1=\widehat{T}_2$ and $i(T_1) \ne i(T_2)$.
\end{lemma}

\begin{proof}
By symmetry, 
we may assume that $T_2$ is obtained from $T_1$ by the left rotation at a node $v$ with shape $(A,(B,C))$ in~$T_1$. 
Let $z$ be the right child of $v$ in $T_1$, which becomes the left child of $v$ in $T_2$. 
Let $a$, $b$, and~$c$ be the kings of the roots of $A$, $B$, and $C$, respectively. Since $A,B,C$ have pairwise disjoint sets of leaves, $a,b,c$ are
pairwise distinct leaves. Since these sets of leaves are unchanged by the rotation, $a,b,c$ do
not depend on whether we compute them in $T_1$ or $T_2$. 

We first show that the rotation only affects the parents of $a$, $b$, and $c$ in the tournament tree. For every node $w\neq z$, the subtree rooted at $w$ has the same set of leaves in~$T_1$ and in~$T_2$ (for $w=v$, it consists of the leaves of $A$, $B$, and $C$), so $\king_{T_1}(w)=\king_{T_2}(w)$, and we denote this leaf by $\king(w)$. For $i\in\{1,2\}$, we have $\king_{T_i}(z)\in\{a,b,c\}$, since it is the king of a child of $z$ in~$T_i$. Hence $\king(v)\in\{a,b,c\}$, since $\king(v)$ is the king of a child of $v$ in~$T_1$. Moreover, $\ROOT(A)$ and $\ROOT(C)$ are the only nodes whose parents change, and the children of $z$ are among $\ROOT(A)$, $\ROOT(B)$, $\ROOT(C)$. Now let $u\notin\{a,b,c\}$ be a leaf. Since $\king_{T_1}(z),\king_{T_2}(z)\in\{a,b,c\}$, we have $X^{T_1}_u=X^{T_2}_u$. 
Its top $a^{T_1}_u=a^{T_2}_u$ is not in $\{z,v,\ROOT(A),\ROOT(B),\ROOT(C)\}$, since the kings of these nodes are in $\{a,b,c\}$ in both trees. So this top has the same parent in $T_1$ and in $T_2$, and the king of this parent is the same in both trees. Hence $u$ has the same parent in $\widehat{T}_1$ and in $\widehat{T}_2$. Similarly, $\king(v)$ has the same parent in $\widehat{T}_1$ and in $\widehat{T}_2$ (if any), since the tops of $X^{T_1}_{\king(v)}$ and~$X^{T_2}_{\king(v)}$ are the same node, which is $v$ or an ancestor of $v$, as nothing changes outside the subtree rooted at $v$. Finally, $\widehat{T}_1$ and $\widehat{T}_2$ have the same root $\king(\{1,2,\ldots,L\})$. 

It remains to compare the parents of the two leaves in $\{a,b,c\}\setminus\{\king(v)\}$. We consider three cases, depending on which of $a,b,c$ is $\king(v)$. 

\medskip{\bf\boldmath Case $\king(v)=a$:} Let $y=\king_{T_1}(z)$, which is in $\{b,c\}$, and let $x$ be the other leaf in $\{b,c\}$. Since $\king(v)=a\neq y$, the node $z$ is the top of $X^{T_1}_y$, and the top of~$X^{T_1}_x$ is the root of $B$ or $C$, whose parent in $T_1$ is $z$. This implies that $\parent_{\widehat{T}_1}(y)=\king(v)=a$ and $\parent_{\widehat{T}_1}(x)=\king_{T_1}(z)=y$. 
Observe that $\king_{T_2}(z)=a$, since 
$a\in\{\king_{T_2}(z),
\king(\ROOT(C))\}$ by 
\zcref{lem:ranking}.
So the tops of $X^{T_2}_b$ and $X^{T_2}_c$ are $\ROOT(B)$ and $\ROOT(C)$, whose parents in $T_2$ are $z$ and $v$, respectively. This implies that $\parent_{\widehat{T}_2}(b)=\parent_{\widehat{T}_2}(c)=a$. Hence $\widehat{T}_2$ is obtained from $\widehat{T}_1$ by the parent slide at $x$ from $y$ to $a$; see \zcref{fig:case-abc} for the case that $y=b$. So $(\widehat{T}_1,\rho)$ and $(\widehat{T}_2,\rho)$ are adjacent in $\mathcal{G}_{\{1,\dots,L\},m}$.

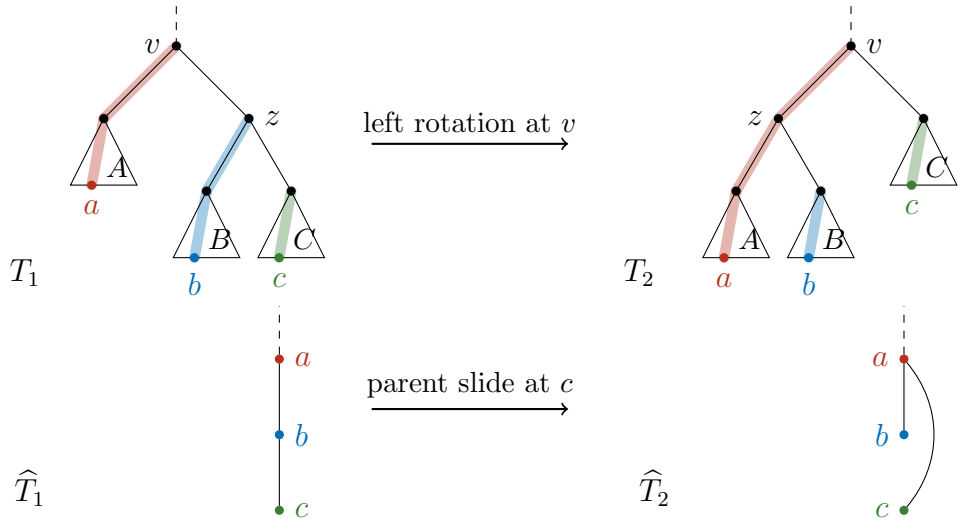
\begin{figure}[!ht]
\[
\begin{array}{ccc}
\begin{tikzpicture}[scale=0.8,baseline=(current bounding box.center)]
  \coordinate (v) at (0,0);    \coordinate (z) at (1.2,-1.2);
  \coordinate (A) at (-1.2,-1.2); \coordinate (B) at (0.5,-2.4); \coordinate (C) at (1.9,-2.4);
  \draw[hl,BrickRed] (A) -- (v);
  \draw[hl,NavyBlue] (B) -- (z);
  \draw[dashed] (v) -- (0,0.8);
  \draw (v) -- (A) (v) -- (z) -- (B) (z) -- (C);
  \tsubtree{-1.2}{-1.2}{A}{BrickRed}{a}
  \tsubtree{0.5}{-2.4}{B}{NavyBlue}{b}
  \tsubtree{1.9}{-2.4}{C}{OliveGreen}{c}
  \node[tnode,label=left:$v$] at (v) {};
  \node[tnode,label=right:$z$] at (z) {};
  \node at (-2.5,-3.75) {$T_1$};
\end{tikzpicture}
&
\begin{tikzpicture}[baseline=-1mm]
  \draw[->,thick] (0,0) -- (2.6,0) node[midway,above]{\small left rotation at $v$};
\end{tikzpicture}
&
\begin{tikzpicture}[scale=0.8,baseline=(current bounding box.center)]
  \coordinate (v) at (0,0);    \coordinate (z) at (-1.2,-1.2);
  \coordinate (C) at (1.2,-1.2); \coordinate (A) at (-1.9,-2.4); \coordinate (B) at (-0.5,-2.4);
  \draw[hl,BrickRed] (A) -- (z) -- (v);
  \draw[dashed] (v) -- (0,0.8);
  \draw (v) -- (z) -- (A) (z) -- (B) (v) -- (C);
  \tsubtree{-1.9}{-2.4}{A}{BrickRed}{a}
  \tsubtree{-0.5}{-2.4}{B}{NavyBlue}{b}
  \tsubtree{1.2}{-1.2}{C}{OliveGreen}{c}
  \node[tnode,label=right:$v$] at (v) {};
  \node[tnode,label=left:$z$] at (z) {};
  \node at (-3.5,-3.75) {$T_2$};
\end{tikzpicture}
\\[4mm]
\begin{tikzpicture}[baseline=(current bounding box.center)]
  \node[tnode,draw=BrickRed,fill=BrickRed,label={[text=BrickRed]right:$a$}] (a) at (0,0) {};
  \node[tnode,draw=NavyBlue,fill=NavyBlue,label={[text=NavyBlue]right:$b$}] (b) at (0,-1) {};
  \node[tnode,draw=OliveGreen,fill=OliveGreen,label={[text=OliveGreen]right:$c$}] (c) at (0,-2) {};
  \draw[dashed] (a) -- (0,0.7);
  \draw (a) -- (b) -- (c);
  \node at (-3.3,-1.7) {$\widehat{T}_1$};
\end{tikzpicture}
&
\begin{tikzpicture}[baseline=-1mm]
  \draw[->,thick] (0,0) -- (2.6,0) node[midway,above]{\small parent slide at $c$};
\end{tikzpicture}
&
\begin{tikzpicture}[baseline=(current bounding box.center)]
  \node[tnode,draw=BrickRed,fill=BrickRed,label={[text=BrickRed]left:$a$}] (a) at (0,0) {};
  \node[tnode,draw=NavyBlue,fill=NavyBlue,label={[text=NavyBlue]left:$b$}] (b) at (0,-1) {};
  \node[tnode,draw=OliveGreen,fill=OliveGreen,label={[text=OliveGreen]left:$c$}] (c) at (0,-2) {};
  \draw[dashed] (a) -- (0,0.7);
  \draw (a) -- (b);
  \draw (c) [bend right=40] to (a);
  \node at (-3.3,-1.7) {$\widehat{T}_2$};
\end{tikzpicture}
\end{array}
\]
\caption{The case $\king(v)=a$ in the proof of \zcref{TT}, when $y=b$ and $x=c$. Top: the left rotation at $v$ from $T_1$ to $T_2$, where the highlighted paths show the parts of $X_a$ (red), $X_b$ (blue), and $X_c$ (green). Bottom: the corresponding change from $\widehat{T}_1$ to $\widehat{T}_2$ is the parent slide at~$c$ from~$b$ to~$a$. }
\label{fig:case-abc}
\end{figure}

\medskip{\bf\boldmath Case $\king(v)=b$:}
By \zcref{lem:ranking}, we have $\king_{T_1}(z)=\king_{T_2}(z)=b$. So for $i\in\{1,2\}$, the tops of $X^{T_i}_a$ and $X^{T_i}_c$ are $\ROOT(A)$ and $\ROOT(C)$, whose parents in $T_i$ are in $\{v,z\}$. This implies that $\parent_{\widehat{T}_1}(a)=\parent_{\widehat{T}_2}(a)=b$ and $\parent_{\widehat{T}_1}(c)=\parent_{\widehat{T}_2}(c)=b$. 
Hence $\widehat{T}_1=\widehat{T}_2$. Consider the corresponding contraction words. 
The nodes $v$ and $z$ are the only nodes whose children change, and both are in~$X^{T_1}_b$ and in~$X^{T_2}_b$. So for every leaf $u\neq b$, we have $X^{T_1}_u=X^{T_2}_u$, and the children of its nodes are the same in $T_1$ and $T_2$, and hence $\omega_u(T_1)=\omega_u(T_2)$. Consider the contraction word of $b$, first in $T_1$ and then in $T_2$. 
In~$T_1$, walking up from $b$, the letter at~$z$ records that the other child $\ROOT(C)$ is to the right, and the letter at $v$ records that the other child $\ROOT(A)$ is to the left, so $b$ gains ``$RL$''. In $T_2$, at $z$ the other child $\ROOT(A)$ is to the left, and at $v$ the other child $\ROOT(C)$ is to the right, so $b$ gains ``$LR$''. 
Writing~$X$ for the (identical) portion of $\omega_b$ from inside $B$ and $Y$ for the (identical) portion from above~$v$, this gives
$\omega_b(T_1)=X\,RL\,Y$ and $\omega_b(T_2)=X\,LR\,Y$; see \zcref{fig:case-b}.
By \zcref{lem:adjtrans},
$\mathrm{inv}(\omega_b(T_1)) = \mathrm{inv}(\omega_b(T_2))+1$, while $\mathrm{inv}(\omega_u(T_1))
= \mathrm{inv}(\omega_u(T_2))$ for every other leaf~$u$. Summing over all~$L$ leaves, $i(T_1) \equiv
i(T_2)+1 \pmod 2$.

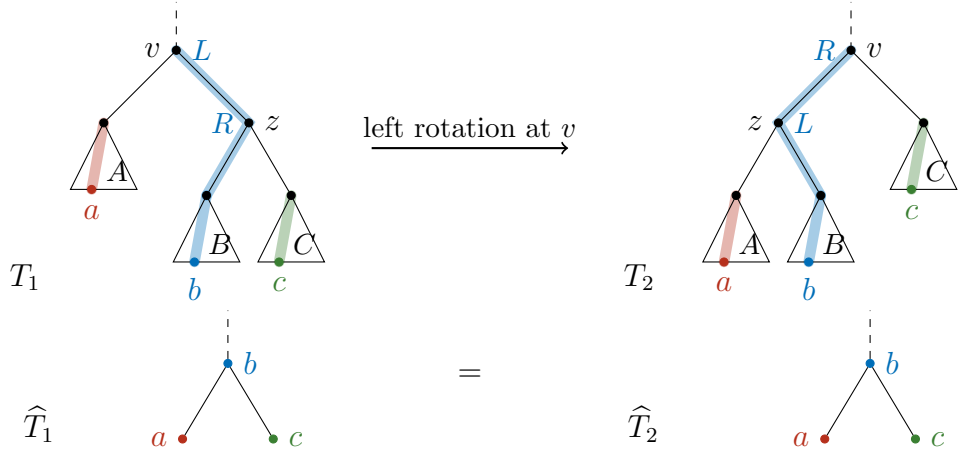
\begin{figure}[!ht]
\[
\begin{array}{ccc}
\begin{tikzpicture}[scale=0.8,baseline=(current bounding box.center)]
  \coordinate (v) at (0,0);    \coordinate (z) at (1.2,-1.2);
  \coordinate (A) at (-1.2,-1.2); \coordinate (B) at (0.5,-2.4); \coordinate (C) at (1.9,-2.4);
  \draw[hl,NavyBlue] (B) -- (z) -- (v);
  \draw[dashed] (v) -- (0,0.8);
  \draw (v) -- (A) (v) -- (z) -- (B) (z) -- (C);
  \tsubtree{-1.2}{-1.2}{A}{BrickRed}{a}
  \tsubtree{0.5}{-2.4}{B}{NavyBlue}{b}
  \tsubtree{1.9}{-2.4}{C}{OliveGreen}{c}
  \node[tnode,label=left:$v$,label={[text=NavyBlue,font=\small]right:$L$}] at (v) {};
  \node[tnode,label={[text=NavyBlue,font=\small]left:$R$},label=right:$z$] at (z) {};
  \node at (-2.5,-3.75) {$T_1$};
\end{tikzpicture}
&
\begin{tikzpicture}[baseline=-1mm]
  \draw[->,thick] (0,0) -- (2.6,0) node[midway,above]{\small left rotation at $v$};
\end{tikzpicture}
&
\begin{tikzpicture}[scale=0.8,baseline=(current bounding box.center)]
  \coordinate (v) at (0,0);    \coordinate (z) at (-1.2,-1.2);
  \coordinate (C) at (1.2,-1.2); \coordinate (A) at (-1.9,-2.4); \coordinate (B) at (-0.5,-2.4);
  \draw[hl,NavyBlue] (B) -- (z) -- (v);
  \draw[dashed] (v) -- (0,0.8);
  \draw (v) -- (z) -- (A) (z) -- (B) (v) -- (C);
  \tsubtree{-1.9}{-2.4}{A}{BrickRed}{a}
  \tsubtree{-0.5}{-2.4}{B}{NavyBlue}{b}
  \tsubtree{1.2}{-1.2}{C}{OliveGreen}{c}
  \node[tnode,label={[text=NavyBlue,font=\small]left:$R$},label=right:$v$] at (v) {};
  \node[tnode,label=left:$z$,label={[text=NavyBlue,font=\small]right:$L$}] at (z) {};
  \node at (-3.5,-3.75) {$T_2$};
\end{tikzpicture}
\\[4mm]
\begin{tikzpicture}[baseline=(current bounding box.center)]
  \node[tnode,draw=NavyBlue,fill=NavyBlue,label={[text=NavyBlue]right:$b$}] (b) at (0,0) {};
  \node[tnode,draw=BrickRed,fill=BrickRed,label={[text=BrickRed]left:$a$}] (a) at (-0.6,-1) {};
  \node[tnode,draw=OliveGreen,fill=OliveGreen,label={[text=OliveGreen]right:$c$}] (c) at (0.6,-1) {};
  \draw[dashed] (b) -- (0,0.7);
  \draw (a) -- (b) -- (c);
  \node at (-2.5,-0.8) {$\widehat{T}_1$};
\end{tikzpicture}
&
\begin{tikzpicture}[baseline=-1mm]
  \node at (1.3,0) {$=$};
\end{tikzpicture}
&
\begin{tikzpicture}[baseline=(current bounding box.center)]
  \node[tnode,draw=NavyBlue,fill=NavyBlue,label={[text=NavyBlue]right:$b$}] (b) at (0,0) {};
  \node[tnode,draw=BrickRed,fill=BrickRed,label={[text=BrickRed]left:$a$}] (a) at (-0.6,-1) {};
  \node[tnode,draw=OliveGreen,fill=OliveGreen,label={[text=OliveGreen]right:$c$}] (c) at (0.6,-1) {};
  \draw[dashed] (b) -- (0,0.7);
  \draw (a) -- (b) -- (c);
\node at (-3,-0.8) {$\widehat{T}_2$};
\end{tikzpicture}
\end{array}
\]
\caption{The case $\king(v)=b$ in the proof of \zcref{TT}. Top: the left rotation at $v$ from $T_1$ to $T_2$, where the highlighted paths show the parts of $X_a$ (red), $X_b$ (blue), and $X_c$ (green); the blue letters at $z$ and $v$ are the corresponding letters of the contraction word $\omega_b$, which change from $RL$ to $LR$. Bottom: the tournament trees are unchanged, $\widehat{T}_1=\widehat{T}_2$. }
\label{fig:case-b}
\end{figure}

\medskip
{\bf\boldmath Case $\king(v)=c$:} Let $y=\king_{T_2}(z)$, which is in $\{a,b\}$, and let $x$ be the other leaf in $\{a,b\}$.
By \zcref{lem:ranking}, 
we have $\king_{T_1}(z)=c$. So the tops of $X^{T_1}_a$ and $X^{T_1}_b$ are $\ROOT(A)$ and $\ROOT(B)$, whose parents in $T_1$ are $v$ and~$z$, respectively. 
This implies that $\parent_{\widehat{T}_1}(a)=\parent_{\widehat{T}_1}(b)=c$. 
Since $\king(v)=c\neq y$, the node $z$ is the top of $X^{T_2}_y$, and the top of $X^{T_2}_x$ is the root of $A$ or $B$, whose parent in $T_2$ is~$z$. This implies that $\parent_{\widehat{T}_2}(y)=\king(v)=c$ and $\parent_{\widehat{T}_2}(x)=\king_{T_2}(z)=y$. 
Hence $\widehat{T}_1$ is obtained from $\widehat{T}_2$ by the parent slide at $x$ from $y$ to $c$. So $(\widehat{T}_1,\rho)$ and $(\widehat{T}_2,\rho)$ are adjacent in~$\mathcal{G}_{\{1,\dots,L\},m}$.
\end{proof}

\section{Finishing the Proof}\label{sec:mainthm}

The \defn{rotation graph} $\mathcal{R}_L$ 
is a graph on all full binary trees with $L$ leaves
such that two such trees  are adjacent whenever one  is obtained from the other by a  rotation.
By \zcref{sec:bijection}, $\mathcal{A}_n$ is isomorphic to $\mathcal{R}_{n-1}$.

\begin{lemma}\label{thm:main2}
For every integer $L\ge2$, 
$ \chi(\mathcal{R}_L) \leq 2\,\chi( \mathcal{G}_{\{1,\dots,L\},\HHH})$
where $\HHH = \lceil \log_2(L+1) \rceil$.
\end{lemma}

\begin{proof}
Let $\phi$ be a $k$-colouring of $\mathcal{G}_{\{1,\dots,L\},\HHH}$. Let $\rho$ be an $\HHH$-ranking of $\{1,\dots,L\}$, which exists by \zcref{lem:nu-king}. 
Let $T$ be a full binary tree with $L$ leaves. 
So $(\widehat{T},\rho)$ is a vertex of~$\mathcal{G}_{\{1,\dots,L\},\HHH}$ by \zcref{lem:PT-proper}. Colour $T$ by the pair $\big(i(T), \phi((\widehat{T},\rho))\big)$. By \zcref{TT}, adjacent trees in~$\mathcal{R}_L$ are assigned distinct colours.  There are $2k$ colours, so $\chi(\mathcal{R}_L) \leq 2k$.
\end{proof}

\begin{lemma}
\label{lem:explicit}
For any finite set $V$ and integer $m\geq 6$, 
\[ \chi(\mathcal{G}_{V,m}) \leq 5\big(\log_2(m-2) - 1\big).\] 
\end{lemma}
\begin{proof}
We proceed by induction on $m\geq 6$. If $m\le10$, then $\chi(\mathcal{G}_{V,m})\le m-1\le5\big(\log_2(m-2)-1\big)$ by \zcref{lem:basepalette}. Now assume $m\ge11$. Then $6\le r_m<m$, since $\binom{5}{2}<m\le\binom{m-1}{2}$. By the definition of $r_m$, we have $\binom{r_m-1}2 \le m-1$, implying 
\[
(r_m-2)^2 \,=\, (r_m-1)(r_m-2) - (r_m-2) \,\le\, 2m-2 -(r_m-2) \,\leq\, 2m-4 \,=\, 2(m-2),
\]
where the last inequality holds since $r_m\ge6$.
Thus
\[
\log_2(r_m-2) -1
< \tfrac12\log_2(2(m-2)) -1
= \tfrac12(\log_2(m-2)-1).
\]
By \zcref{lem:NewLemma} and induction hypothesis, 
\[
\chi(\mathcal{G}_{V,m}) 
\le
2\,\chi(\mathcal{G}_{V,r_m}) 
\le 2\cdot5\big(\log_2(r_m-2)-1\big) 
 <
5\big(\log_2(m-2)-1\big).\qedhere
\]
\end{proof}

\begin{proof}[Proof of \zcref{thm:main}] 
Let $\HHH = \lceil\log_2 n\rceil\geq 2$. Since $\mathcal{A}_n \cong \mathcal{R}_{n-1}$, by \zcref{thm:main2} applied with $L=n-1$ (so that $\lceil\log_2(L+1)\rceil=\HHH$), 
\[\chi(\mathcal{A}_n) 
\,=\, \chi(\mathcal{R}_{n-1}) 
\,\le\, 2\,\chi(\mathcal{G}_{\{1,\dots,n-1\},\HHH}) .\] 
If $\HHH\in\{2,\dots,5\}$ then by \zcref{lem:basepalette},  
\[\chi(\mathcal{A}_n) 
\,\le\, 2\,\chi(\mathcal{G}_{\{1,\dots,n-1\},\HHH}) 
\,\le\, 2(\HHH-1) \,\le\, 8 < 10 \log_2\log_2 n,\] 
since $n\ge4$ gives $\log_2\log_2 n\ge1$. 
Now assume that $\HHH\ge 6$. By \zcref{lem:explicit},
\[
\chi(\mathcal{A}_n) \,\le\, 
2\,\chi(\mathcal{G}_{\{1,\dots,n-1\},\HHH}) \,\le\, 10\big(\log_2(\HHH-2)-1\big)  \,<\, 10\log_2\log_2 n,
\]
since $\HHH-2<\log_2 n-1<2\log_2 n$.\qedhere
\end{proof}

\section{Open Problem}\label{sec:open}

The main question left open by \zcref{thm:main} is whether $\chi(\mathcal{A}_n)$ is bounded from above by an absolute constant independent of $n$. Two standard lower bounds do not help here. 
First,~$\mathcal{A}_n$ is triangle-free for $n\ge4$~\cite{FabilaEtAl2009}, so the clique-number lower bound for the chromatic number is at most two. 
Second,~$\mathcal{A}_n$ contains no $K_{2,3}$ subgraph\footnote{If distinct triangulations $T$ and $T'$ have a common neighbour, then they are non-adjacent since $\mathcal{A}_n$ is triangle-free, so $T'=(T\setminus\{d_1,d_2\})\cup\{e_1,e_2\}$ for some diagonals $d_1,d_2,e_1,e_2$, and every common neighbour of $T$ and $T'$ is of the form $(T\setminus\{d_i\})\cup\{e_j\}$. Since each diagonal of a triangulation can be flipped in only one way, each $d_i$ gives at most one common neighbour, so $T$ and $T'$ have at most two common neighbours. Hence $K_{2,3}$ is not a subgraph of $\mathcal{A}_n$.}. By the zig-zag theorem of Simonyi and Tardos~\cite{ST06}, every graph that is topologically $t$-chromatic in their sense contains $K_{\lceil t/2\rceil,\lfloor t/2\rfloor}$ as a subgraph, so $\mathcal{A}_n$ is not topologically $5$-chromatic. In particular, the topological lower bound of Lov\'asz~\cite{Lovasz1978}, which bounds $\chi(G)$ from below in terms of the connectivity of the neighbourhood complex of $G$, is at most $4$ for $\mathcal{A}_n$, since every graph for which this bound is at least $t$ is topologically $t$-chromatic (see~\cite{MZ04,ST06}). Therefore these methods cannot show that $\chi(\mathcal{A}_n)\ge5$. 
According to \cite{AddarioBerryEtAl}, Fabila-Monroy showed that $\chi(\mathcal{A}_{10})\ge 4$. We computationally verified\footnote{
  Verification codes and documents are available 
  in the ancillary files on arXiv.
} that $\chi(\mathcal{A}_{15})\le 4$. 
Since $\mathcal{A}_n$ is isomorphic to an induced subgraph of $\mathcal{A}_{n+1}$, it follows that $\chi(\mathcal{A}_n)=4$ for $n\in\{10,\dots,15\}$. It remains open whether $\chi(\mathcal{A}_n)\ge5$ for some~$n$.

\subsection*{Statement of AI use} GPT-6 Astra and Claude were used in the discovery of this proof, and for the creation of the first draft of the paper. 
The authors have rewritten the paper, declare that they understand the proof, and take responsibility for the entire paper.

\newcommand{\etalchar}[1]{$^{#1}$}

\end{document}